\documentclass[11pt, a4paper]{amsart}

\usepackage{amsfonts}
\usepackage{amssymb}
\usepackage{amsmath, amsthm,color}
\usepackage{hyperref}
\usepackage{pdfsync}
\usepackage{float}
\usepackage{graphicx}
\usepackage{color}
\usepackage{bbm, dsfont}
\usepackage[margin=0.96in]{geometry}

\newcommand{\La}{\langle}
\newcommand{\Ra}{\rangle}

\def\supp{\operatorname{supp}}

\newcommand{\1}{{\bf 1}}

\newcommand{\eps}{\varepsilon}
\newcommand{\la}{\lambda}

\newcommand{\om}{\omega}

\newcommand{\cD}{{\mathcal D}}
\newcommand{\cA}{{\mathcal A}}
\newcommand{\cR}{{\mathcal R}}
\newcommand{\cS}{{\mathcal S}}

\newcommand{\cB}{{\mathcal B}}
\newcommand{\cI}{{\mathcal I}}

\newcommand{\bR}{{\mathbb R}}

\newtheorem{theorem}{Theorem}[section]

\newtheorem{lemma}[theorem]{Lemma}

\newtheorem{remark}[theorem]{Remark}

\numberwithin{equation}{section}

\begin{document}
\title[Martingale transform and its sharp lower estimate]{Martingale transform and Square function: some weak and restricted weak sharp weighted estimates}
\author[Paata Ivanisvili, \,\, Alexander Volberg]{Paata Ivanisvili,\,\, Alexander Volberg}
\thanks{ Volberg is partially supported by the NSF DMS-1600065.  This paper is  based upon work supported by the National Science Foundation under Grant No. DMS-1440140 while two of the authors were in residence at the Mathematical Sciences Research Institute in Berkeley, California, during the Spring 2017 semester. }
\address{Department of Mathematics, Princeton University, Princeton, NJ, USA}
\address{Department of Mathematics,  UC Irvine, Irvine, CA, USA}
\email{ivanishvili.paata@gmail.com \textrm{(P.\ Ivanisvili)}}
\address{Department of Mathematics, Michigan State University, East Lansing, MI 48823, USA}
\email{volberg@math.msu.edu \textrm{(A.\ Volberg)}}
\makeatletter
\@namedef{subjclassname@2010}{
  \textup{2010} Mathematics Subject Classification}
\makeatother
\begin{abstract} 
Following the ideas of \cite{LNO} we prove that there is a sequence of weights $w\in A^d_1$ such that $[w]^d_{A_1}\to \infty$, and martingale transforms $T$
such that
 $\|T: L^1(w) \to L^{1, \infty}(w)\| \ge c [w]^d_{A_1}\log [w]^d_{A_1}$ with an absolute positive $c$.
We also show the existence of the sequence of weights (now in $A_2$) such that $[w]^d_{A_2}\to \infty$, and such that 
  we have

\textup(1\textup) $[w]_{A_2^d}\asymp \|M^d\|_{w^{-1}}^2$;

\textup(2\textup)  $\|S_{w}: L^{2} (w) \to L^2(w^{-1})\| \ge c\, \|M^d\|_{w^{-1}}\sqrt{\log \|M^d\|_{w^{-1}}}$;

\textup(3\textup)
$
\|S_{w}: L^{2,1} (w) \to L^2(w^{-1})\| \ge c\,  \|M^d\|_{w^{-1}}\sqrt{\log \|M^d\|_{w^{-1}}};
$

\textup(4\textup)
 $\|S: L^2(w)\to L^{2, \infty}(w)\|=\|S_{w^{-1}}: L^{2}(w^{-1}) \to L^{2,\infty}(w)\|\le  C\, \|M^d\|_{w^{-1}}\le C\, ([w]^d_{A_2})^{1/2}$.

 \end{abstract}
\maketitle 

\section{Introduction}
\label{intro}


In \cite{R} Reguera answered in the negative the following question of Muckenhoupt: is it true that the martingale transform will be weakly bounded from $L^1(w)$ to $L^{1, \infty}(Mw)$, where $M$ means Hardy-Littlewood maximal function.

In their paper \cite{RT} Reguera and Thiele answered in the negative the following question of Muckenhoupt: is it true that the Hilbert transform  will be weakly bounded from $L^1(w)$ to $L^{1, \infty}(Mw)$.

Their construction gives a very irregular weight $w$ (a sort of sophisticated sum of delta functions), so their weight was not in Hunt-Muckenhoupt--Wheeden class $A_1$. In particular, another problem of Muckenhoupt remained still open even after \cite{R}, \cite{RT}: can one give the linear estimate 
of the norm:
\begin{equation}
\label{Tlin}
\|T: L^1(w)\to L^{1, \infty}(w) \|\le C[w]_{A_1}\,?
\end{equation}

The estimate 
\begin{equation}
\label{LOP}
\|T: L^1(w)\to L^{1, \infty}(w) \|\le C[w]_{A_1}\log [w]_{A_1}
\end{equation}
was obtained in \cite{LOP1}, \cite{LOP2}.
But the question was whether the logarithmic term can be dropped.

The Bellman function construction in \cite{NRVV1}, \cite{NRVV2} showed that linear estimate is false, in fact, in these preprints  a sequence of  weights 
$w$, was proved to exist such that $[w]_{A_1}\to \infty$ but $\|T: L^1(w)\to L^{1, \infty}(w) \|\ge c[w]_{A_1}(\log [w]_{A_1})^{1/3}$.
Operator $T$ was either a martingale transform, or the dyadic shift, or the Hilbert transform.

Finally \cite{LNO}  proves the existence of  a sequence of  weights 
$w, [w]_{A_1}\to \infty$, such that  $\|T: L^1(w)\to L^{1, \infty}(w) \|\ge c[w]_{A_1}\log [w]_{A_1}$. Operator $T$ is the Hilbert transform in \cite{LNO}.

Using the idea of \cite{LNO}, here we prove the existence of  the weights, which will satisfy the same properties as above, but for $T$ being the martingale transform. We also consider questions related to the estimate of the square functions given in \cite{DSaLaRey}.

\section{Construction of weights and special intervals}
\label{spI}

We follow \cite{LNO} in an almost verbatim fashion. As our goal is to repeat the proof of \cite{LNO} for the martingale transform, the main issue is to choose the signs of the martingale transform. This should be done consistently simultaneously for all points, where we estimate the transform from below.

For dyadic interval $I$ we denote $I^-, I_+$ its left and right children.  We also denote
$$
I_0=I, I_1= I^{++}, I_2= I_1^{++},\dots. I_{m-1} = I_{m-2}^{++},\, m=2, \dots, k,
$$
where we put 
$$
\eps= 4^{-k}\,.
$$
We fix a large number $p$ (will be $\asymp 1/\eps$), and we build the sequence of weights by the rule:
let $\om, \sigma$ two numbers such that $\om\sigma=p$
$$
w_0(\om, \sigma, I) = \frac{\omega}{\sqrt{p}}\left( (\sqrt{p}-\sqrt{p-1})\chi_{I_-} + (\sqrt{p}+\sqrt{p-1})\chi_{I_+} \right) \,,
$$
\begin{equation}
\label{wn}
w_n(\om,\sigma, I) = \sum_{m=0}^{k-2} w_{n-1} (3\om, \sigma/3, I_m^{+-}) +\frac{\om}{p}\left( \sum_{m=0}^{k-2} \chi_{I_m^{-}} + \chi_{I_{k-1}^-} + \tau(\eps)\chi_{I_{k-1}^+} \right),
\end{equation}
where $\tau(\eps)=\frac{9\eps}{1+5\eps}$ we find from the following
\begin{figure}
\includegraphics[scale=1]{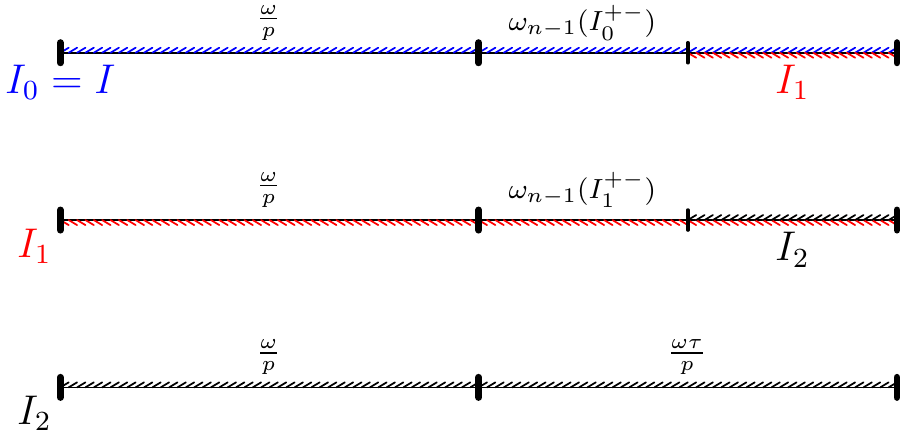}
\caption{Construction of $w_{n}$ when $k=3$.}
\label{fig:dom}
\end{figure} 
\begin{lemma}
\label{av}
$\La w_n(\om, \sigma, I)\Ra_I =\om,\,\, \La w_n^{-1}(\om, \sigma, I)\Ra_I  =\sigma.$
\end{lemma}
\begin{proof}
By induction. For $n=0$ it is clear.
Let it be proved for $n-1$. To prove for $n$, we  notice that
value $\frac{\om}{p}$ happens first of all on measure $\frac12(|I_0|+\dots |I_{k-2}|) = \frac1{2}( 1+1/4+\dots+ 1/4^{k-2})|I|$, that is on $1/2\frac{1-4\eps}{3/4}|I|$, that is
$\frac23(1-4\eps)|I|$. But the same value $\frac{\om}{p}$ happens on measure $\frac12 |I_{k-1}| =2\eps|I|$. So  $\frac{\om}{p}$ totally  assigned to
measure $\frac23 (1-4\eps + 3\eps)|I| = \frac23(1-\eps)|I|$. Value $\frac{\om \tau}{p}$ is assigned to measure $2\eps|I|$. We are left with $(\frac13-\frac43\eps)|I|$. Using induction hypothesis we get that the average of $w_n$ over $I$ is
$$
3(\frac13-\frac43\eps)  +\frac23(1-\eps)\frac1p +2\eps\frac{\tau}{p} =1,
$$which gives the first equation on $\tau, p$:
\begin{equation}
\label{1}
\left( \frac23(1-\eps) +2\eps\tau\right)\frac1p -4\eps=0\,.
\end{equation}
Now we do the same with weight 
$
w_n^{-1}:
$
\begin{equation}
\label{2}
\frac{\sigma}{3} (\frac13-\frac43\eps)  + \sigma \frac23(1-\eps) + \frac{\sigma}{\tau} 2\eps =\sigma\,.
\end{equation}
Equations \eqref{1}, \eqref{2} give
\begin{equation}
\label{3}
\tau=\frac{9\eps}{1+5\eps},\,\, p \approx  \frac1{6\eps}\,.
\end{equation}
\end{proof}

In what follows $n$ will be chosen
\begin{equation}
\label{n}
n=4^k\,.
\end{equation}

Notice that
\begin{equation}
\label{p}
p  \asymp 4^k\,.
\end{equation}


\begin{lemma}
\label{Qp2}
We have $[w_{n}(\omega, \sigma, I)]_{A^{d}_{2}(I)} \asymp p^2$. 
\end{lemma}
\begin{proof}
Consider the  dyadic  interval $I_{k-2}$. 
We have 
\begin{align*}
&\langle w_{n} \rangle_{I_{k-2}}=\frac{1}{2}\cdot \frac{\omega}{p}+\frac{1}{4} \cdot 3\omega +\frac{1}{8}\cdot \frac{\omega}{p} + \frac{1}{8}\cdot \frac{\omega \tau }{p};\\
&\langle \omega_{n}^{-1}\rangle_{I_{k-2}} = \frac{1}{2}\cdot \frac{p}{\omega}+\frac{1}{4} \cdot \frac{\sigma}{3} +\frac{1}{8}\cdot \frac{p}{\omega} + \frac{1}{8}\cdot \frac{p}{\omega \tau}.
\end{align*}
Therefore 
$$
[w_{n}]_{A_{2}^{d}(I)}\geq \langle w_{n} \rangle_{I_{k-2}}\langle \omega_{n}^{-1}\rangle_{I_{k-2}} > \frac{p}{\tau} \frac{4}{32} \asymp p^{2}.
$$
The reverse inequality $[w_{n}]_{A_{2}^{d}(I)}\leq Cp^{2}$ follows from Theorem~\ref{Md}. 
\end{proof}

Intervals of type $I_{k-1}^+$ play special role.  We call them special. Assume that $I$ is an interval involved in forming $\om_{n-\ell}$. Then there is only one 
special interval in $I\setminus \supp w_{n-\ell-1}$.  Its length is $2 \frac1{4^k} |I|$. But there are $k-2$ such special intervals in $I\cap (\supp w_{n-\ell-1} \setminus \supp w_{n-\ell-2})$. Their  total length is  
$$
2\frac1{4^k}\sum_{m=1}^{k-2} \frac1{4^m} |I|= \frac13\big(1-\frac1{4^{k-2}}\big)\frac2{4^k}|I|\,.
$$
Similarly, the length of the union of special intervals in $I\cap (\supp w_{n-\ell-2} \setminus \supp w_{n-\ell-3})$ is
$$
\left(\frac13\big(1-\frac1{4^{k-2}}\big)\right)^2\frac2{4^k}|I|
$$
et cetera.

\medskip

If we denote the family of such  special intervals  in 
$$
[0,1) \cap (\supp w_{n-\ell}\setminus \supp w_{n-\ell-1}),
$$
by $\cA_\ell$, and their union by by $A_\ell$, we then have
\begin{equation}
\label{A}
|A_\ell | = \left(\frac13\big(1-\frac1{4^{k-2}}\big)\right)^\ell \,\frac2{4^k}\,.
\end{equation}

\section{Martingale transform estimate}
\label{mart}
We are going to find $A_1^d$ weights such that
$$
\|T: L^1(w) \to L^{1, \infty} (w)\| \ge c [w]_{A_1^d} \log [w]_{A_1^d}\,.
$$

It is enough to construct weights $w\in A_2$ such that maximal function 
$$
\|M^d: L^2(w^{-1})\to L^2(w^{-1}) \|\le A \,p,
$$
but
\begin{equation}
\label{log1}
\|T(\chi_{I_0} w)\|^2_{w^{-1}} \ge c p^2 (\log p)^2\|\chi_{I_0}\|_w\,.
\end{equation}
This is explained in \cite {LNO} by using the extrapolation of Rubio de Francia. The reader may find this explanation repeated below during the proof of Theorem \ref{log}.

\medskip

We choose the same sequence of weights as in Section \ref{spI}.
Now consider a collection of special intervals $J$ as was introduced in Section \ref{spI}. This family  splits into $\cA_\ell$ collections, $A_\ell=\cup_{J\in \cA_\ell}J$.
Let $J\in \cA_\ell$, $x\in J$. First we want to estimate from below 
$$
T_Jw_n(x) = \sum_{R\in \cD: R\in row(J)} \eps_R(w_n, h_R)h_R (x), \,x\in J\,.
$$
Let us explain what is $row(J)$. Let the interval forming $w_{n-\ell}$ and containing $J$ be called $K$. Then $J= K_{k-1}^{+}$. Consider also $K=K_0, K_1=K^{++}, \dots, K_m,\dots, K_{k-2}$.
$J$ is the right child of $K_{k-1}$. In the sum forming $Tw_n(x), x\in J,$ above we choose first $R=K^{+}, K_1^{+},\dots, K_m^{+}, \dots, K_{k-2}^{+}$. 
Call them $R(J)=\{K^{+}(J), K_1^{+}(J),\dots, K_m^{+}(J), \dots, K_{k-2}^{+}(J)\}$.

\medskip

We will refer to this collection $R(J)=\{K^{+}(J), K_1^{+}(J),\dots, K_m^{+}(J), \dots, K_{k-2}^{+}(J)\}$ as the {\it row} of intervals, this is $row(J)$.

\medskip

We did not choose yet the signs $\eps_{R}$. Here is the choice
\begin{equation}
\label{ch}
\forall J\in \cA_\ell \,\,\text{special}, \, \text{and for even} \,\ell,\, \,\eps_{R} =-1, \,\, \text{if} \,R\in R(J),\,\,\text{otherwise}\,\, \eps_R =0\,.
\end{equation}
In other words, intervals $I$ in the union of all rows of all special intervals in $\cA_\ell$ with even $ \ell$, got $\eps_I=-1$, all other intervals $I\in \cD$ got $\eps_I=0$.

\medskip

Recall that we fixed an interval $J\in \cA_\ell$, $ \ell$ being even, and $x\in J$. We have then
$$
(w_n, h_{K_m^{+}(J)})h_{K_m^{+}(J)}(x)= (\La w_n\Ra_{K_m^{+-}(J)} - \La w_n\Ra_{K_m^{++}(J)}),\, x\in J\,.
$$

But interval $K_m^{+-}$ is by construction a forming interval of $w_{n-\ell-1}$, thus by \eqref{wn} the average over it is $3^{\ell+1}\omega$. On the other hand,
the average $\La w\Ra_{K_m^{++}}$ will be some average of $\frac{\om}{p}$ and of 
$3^{\ell+1}\omega$ on $k-2-m$ other intervals forming $w_{n-\ell-1}$ and lying in $K$ 
to the right of $K_m^{+-}$.  Their total mass is at most $\frac13 |K_m^{+-}|$. Thus, 
this second average is fixed small constant smaller than  $3^{\ell+1}\omega$. Thus, with positive absolute constant $c$
\begin{equation}
\label{KmJ}
(w, h_{K_m^{+}(J)})h_{K_m^{+}(J)}(x)\ge  c_1 3^{\ell+1} \om\,.
\end{equation}
Hence, if  $\ell$ is even we will have positive contributions of order $3^{\ell+1}$ from its row $row(J)$.
Therefore,
\begin{equation}
\label{TJe}
x\in J, \, J\in \cA_\ell\Rightarrow T_Jw_n(x) \ge  ck 3^{\ell+1} \om\,,
\end{equation}
where one can see that $c=\frac23-O(\frac1k)$.

\medskip

Now we need to bookkeep the contribution of $T_{\tilde J} w_n (x)$ at the same point $x\in J$, where we need 
to take into account all special $\tilde J\neq J$. This contribution is formed by  intervals $\tilde J\in \cA_{\ell'}$, $\ell'$ is even, $\ell' =\ell-2, \ell-4,\dots, 0$. All other contributions are zero (for $x$ in this fixed $J$).

As an example, consider the tower of intervals
 $J\subset K_{m_\ell}^{+-}(\tilde J_\ell)\subset  \dots \subset K_{m_0}^{+-}(\tilde J_0)$. 
Interval $ K_{m_{\ell'}}^{+-}(\tilde J_{\ell'})$ is a forming interval of $w_{n-\ell'-1}, \,\ell'<\ell$, $\ell'$ is even. 
But the contribution will be not only from this tower, but also from all
the intervals lying in the same rows as the intervals in the tower above.

The contribution to $Tw_n(x), x\in J$, of the rows assigned to other special intervals is zero. 

We need to consider only the contribution of the rows of intervals to which the intervals in the tower above belong. This contribution will be (by absolute value) at most
$$
k(3^{\ell+1-2}+ 3^{\ell+1-4}+\dots)\le \frac12 k 3^\ell\,,
$$
and, therefore, the total contribution of those $\tilde J$'s sums to at most $|\sum_{i=0, i\, even}^{\ell -2}T_{\tilde J_i}w_n (x)|\le \frac12 k 3^\ell$, $x\in J$,  and
 cannot spoil the number $ c k 3^{\ell+1} \om$ from \eqref{TJe}.


\medskip


Also $w^{-1}_n$-measure of such a $J\in \cA_\ell$ as above is $\frac{1+\eps}{9\eps} \frac{\sigma}{3^\ell} |J|$.
Combining this, \eqref{TJe},   and   estimate of $|A_\ell|$  from \eqref{A}, we get ($\ell$ is even)
$$
\int_{A_\ell} (T w_n)^2 w^{-1}_n dx \ge c\om^2 k^2 3^{2\ell} \frac{1+\eps}{9\eps} \frac{\sigma}{3^\ell}\left(\frac13 \big(1-\frac1{4^{k-2}}\big)\right)^\ell \frac1{4^k}\,.
$$
Or, using that $\om\sigma=p,  4^k =1/\eps$, we get
$$
\int_{A_\ell} (T w_n)^2 w^{-1}_n dx \ge c\om \, k^2 \,p \frac{1+\eps}{9\eps} \left(\big(1-\frac1{4^{k-2}}\big)\right)^\ell \frac1{4^k}\ge c \om \, k^2\, p \left(\big(1-\frac1{4^{k-2}}\big)\right)^\ell\,.
$$
Now,
$$
\int_0^1(T w_n)^2 w^{-1}_n dx \ge  c \om \, k^2\, p \sum_{\ell=0,\, \ell\, even}^{4^k}\left(\big(1-\frac1{4^{k-2}}\big)\right)^\ell \ge c\om \, k^2\, p 4^k\,.
$$
But $4^k\approx p, k\approx \log p$. So we get
\begin{equation}
\label{Sbelow}
\int_0^1(T w_n)^2 w^{-1}_n dx \ge  c \om \, p^2 (\log p)^2=c\,  p^2 (\log p)^2 \int_0^1 w_n dx\,.
\end{equation}

\begin{theorem}
\label{Md}
With a finite absolute constant $A$ one has $\|M^d: L^2(w_n^{-1})\to L^2(w_n^{-1})\| \le A\, p$.
\end{theorem}

We prove it in Section \ref{MAX}, the proof is the same as in \cite{LNO}, only easier because it is dyadic.

Given  Theorem \ref{Md} and what was done above  we can  now prove the following theorem, the analog of the main result of \cite{LNO}, but for the martingale transform instead of the Hilbert transform.

\begin{theorem}
\label{log}
There is a sequence of weights $w\in A^d_1$ such that $[w]^d_{A_1}\to \infty$, and martingale transforms $T$
such that
 $\|T: L^1(w) \to L^{1, \infty}(w)\| \ge c [w]^d_{A_1}\log [w]^d_{A_1}$ with an absolute positive $c$.
 \end{theorem}
 
 The proof of Theorem \ref{log} is verbatim the same as in \cite{LNO}, the corresponding weights $W_n\in A_1^d$ 
 are obtained from $w_n\in A_2^d$ constructed above by the method  of Rubio de Francia. For these weights
 one uses the same  dyadic martingale transforms as above. We repeat the proof for the convenience of the reader. The proof also shows how complicated are weights $W_n$ built with the help of $w_n$ from above.

\begin{proof}
We think of $T$ as of linear operator with nice kernel acting in all spaces with Lebesgue measure on 
$I_0$, in practice $T$ will be, say, a martingale transform with only finitely many $\eps_I$ non-zero, 
but the estimates will not depend on how many non-zeros $T$ has. 

We always think that $T$ has symmetric or anti-symmetric kernel.

By $T_w$ we understand the operator which acts on test functions as follows:
$$
T_w f= T (w f)\,.
$$
Let $w$ be a weight (for our goals it will be one of $w_n$ built above), and $\alpha>0$.
Let $g$ be a function from $L^2(w^{-1})$ to be chosen soon, $\|g\|_{w^{-1}}=1$, and  we use Rubio de Francia function
$$
\cR g =\sum_{k=0}^\infty\frac{(M^d)^k g}{2\|(M^d)^k\|_{L^2(w^{-1})}}\,.
$$
Then 
\begin{equation}
\label{Rg}
\|Rg\|_{w^{-1}}\le 2,\,\, g\le Rg, \,\, [Rg]_{A_1^d} \le \|M^d : L^2(w^{-1})\to L^2(w^{-1})\|\,.
\end{equation}
Then choosing appropriate  $g$, $\|g\|_{w^{-1}}=1$, we can write ($F=wf$) by using \eqref{Rg}:
\begin{align}
\label{cR}
& \alpha(w\{T_{w^{-1}} F > \alpha\})^{1/2} = \alpha(w\{T f > \alpha\})^{1/2} \le 2\alpha\int_{Tf >\alpha} g dx\le  \nonumber\\
&  2\alpha\int_{Tf >\alpha} \cR g dx \le 2N([\cR g]_{A_1^d}) \int |f| \cR g dx \le 2N([\cR g]_{A_1^d}) \|f\|_w\|\cR g\|_{w^{-1}} \le \nonumber\\
& 4 N([\cR g]_{A_1^d}) \|f\|_w= 4 N([\cR g]_{A_1^d}) \|F\|_{w^{-1}}\,.
 \end{align}
 Here $N([\cR g]_{A_1^d}) $ denotes the estimate from above of the weak norm $T: L^1(Rg)\to L^{1,\infty}(Rg)$.
 
 \medskip
 
 This weight $Rg$ is the future weight $W_n$ mentioned before the starting of the proof.
 
 \medskip
 
 Henceforth, one obtains the estimate
 \begin{equation}
 \label{N}
 N([\cR g]_{A_1^d}) \ge \frac14 \| T_{w^{-1}}: L^2(w^{-1})\to L^{2, \infty} (w)\|\,.
 \end{equation}
 By duality and (anti) symmetry of $T$, the latter norm is
  \begin{equation}
 \label{duality}
 \| T_{w^{-1}}: L^2(w^{-1})\to L^{2, \infty} (w)\|  = \|T_w: L^{2,1}(w)\to L^{2}(w^{-1})\|\ge \frac{\|T(\chi_{I_0} w)\|_{w^{-1}}}{\|\chi_{I_0}\|_w}\,.
 \end{equation}
In the last inequality we used the fact that in the norms of characteristic function in $L^{2,1}(w)$ and in $L^2(w)$  are the same.

\medskip

Use \eqref{Rg} and \eqref{N}, \eqref{duality}:
\begin{equation}
\label{NMd}
\frac{\|T(\chi_{I_0} w)\|_{w^{-1}}}{\|\chi_{I_0}\|_w}\le 4 N(\|M^d\|_{w^{-1}})\,.
\end{equation}

Now we plug into this inequality (with $w=w_n$) inequality \eqref{log1} and the result of Theorem \ref{Md}. Then we obtain
$$
p\log p \le C N(p)\,.
$$
This is what we wanted.

\end{proof}

\section{Square function}
\label{sqF}

 Let $J$ be one of intervals of $A_\ell$, $x\in J$. 

We want to estimate from below ($\cD$ is the dyadic lattice in $[I_0:=0,1)$)
$$
S^2w_n(x) \ge \sum_{R\in \cD: J\subset R} (w_n, h_R)^2\frac{\chi_R(x)}{|R|}\,.
$$
Let the interval forming $w_{n-\ell}$ and containing $J$ be called $K$. Then $J= K_{k-1}^{+}$. Consider also $K=K_0, K_1=K^{++}, \dots, K_m,\dots, K_{k-2}$.
$J$ is the right child of $K_{k-1}$. In the sum forming $S^2w(x), x\in J,$ above we choose only $R=K^{+}, K_1^{+},\dots, K_m^{+}, \dots, K_{k-2}^{+}$.

Then
$$
(w_n, h_{K_m^{+}})^2\frac{\chi_{K_m^{+}}(x)}{|K_{m^{+}}|}\ge c(\La w_n\Ra_{K_m^{+-}} - \La w_n\Ra_{K_m^{++}})^2,.
$$

But interval $K_m^{+-}$ is by construction a forming interval of $w_{n-\ell-1}$, thus by \eqref{wn} the average over it is $3^{\ell+1}\omega$. On the other hand,
the average $\La w\Ra_{K_m^{++}}$ will be some average of $\frac{\om}{p}$ and of $3^{\ell+1}\omega$ on $k-2-m$ other intervals forming $w_{n-\ell-1}$ and lying in $K$ to the right of $K_m^{+-}$.  Their total mass is at most $\frac13 |K_m^{+-}|$. Thus, this second average is fixed small constant smaller than  $3^{\ell+1}\omega$. Thus, with positive absolute constant $c$
\begin{equation}
\label{Km}
(w, h_{K_m^{+}})^2\frac{\chi_{K_m^{+}}(x)}{|K_{m^{+}}|}\ge c 3^{2\ell} \om^2\,.
\end{equation}
Hence,
\begin{equation}
\label{S2J}
x\in J, \, J\subset A_\ell\Rightarrow S^2w(x) \ge c k 3^{2\ell} \om^2\,.
\end{equation}
Also $w^{-1}_n$-measure of such a $J$ as above is $\frac{1+\eps}{9\eps} \frac{\sigma}{3^\ell} |J|$.
Combining this, \eqref{S2J} and   estimate of $|A_\ell|$  from \eqref{A} we get
$$
\int_{A_\ell} (S^2 w_n) w^{-1}_n dx \ge c\om^2 k 3^{2\ell} \frac{1+\eps}{9\eps} \frac{\sigma}{3^\ell}\left(\frac13 \big(1-\frac1{4^{k-2}}\big)\right)^\ell \frac1{4^k}\,.
$$
Or, using that $\om\sigma=p,  4^k =1/\eps$, we get
$$
\int_{A_\ell} (S^2 w_n) w^{-1}_n dx \ge c\om \, k \,p \frac{1+\eps}{9\eps} \left(\big(1-\frac1{4^{k-2}}\big)\right)^\ell \frac1{4^k}\ge c \om \, k\, p \left(\big(1-\frac1{4^{k-2}}\big)\right)^\ell\,.
$$
Now,
$$
\int_0^1(S^2 w_n) w^{-1}_n dx \ge  c \om \, k\, p \sum_{\ell=0}^{4^k}\left(\big(1-\frac1{4^{k-2}}\big)\right)^\ell \ge c\om \, k\, p 4^k,.
$$
But $4^k\approx p, k\approx \log p^2$. So we get
\begin{equation}
\label{Sbelow}
\int_0^1(S^2 w_n) w^{-1}_n dx \ge  c \om \, p^2 \log p^2=c\,  p^2 \log p \int_0^1 w_n dx\,.
\end{equation}

Recall that  the Lorentz space $L^{2,1}(\mu)$ is characterized by the following norm:
$$
\|f\|_{L^{2,1}(\mu)}:=\int_0^\infty (\mu\{x: |f(x)|>t\})^{1/2} dt\,.
$$
It is smaller than $L^2(w)$ of course. Its dual is $L^{2, \infty}(\mu)$ and the norm of characteristic functions in $L^{2,1}(\mu)$ and in $L^2(\mu)$ coincide.

\medskip

Let us denote 
\begin{align*}
&n:= \|S_w: L^{2,1} (w)\to L^2(w^{-1})\|,\\
&N:= \|S_w: L^{2} (w)\to L^2(w^{-1})\|\,.
\end{align*}

  \medskip

  Here is a comparison of norms of $M_d$ and $S$ in weighted spaces, which seems to be new.
  \begin{theorem}
  \label{SM}
There exists a positive absolute constant $c$ and sequence of weights $w$ in $A_2^d$, such that $[w]_{A_2}^d\to \infty$ and at the same time
\begin{equation}
\label{comp}
\|S_w\|_{L^2(w)\to L^2(w^{-1})}=\|S\|_{w^{-1}} \ge  c \|M^d\|_{w^{-1}}\sqrt{\log \|M^d\|_{w^{-1}}}\,.
\end{equation}
\end{theorem}
\begin{proof}
Let us combine \eqref{Sbelow} and Theorem \ref{Md}.
\end{proof}


Now let us  keep in mind Lemma \ref{Qp2}.
Notice that in the right hand side we have (here $I_0:=[0,1)$)
$$
\int_0^1 w_n dx  =\|\chi_{I_0} \|^2_{w_n} = \|\chi_{I_0} \|^2_{L^{2,1}(w_n)}\,.
$$
In fact, $L^2(w)$ and $L^{2,1}(w)$ norms are equivalent on characteristic functions.
Hence, \eqref{Sbelow} demonstrates the following sharpening of Theorem \ref{SM}:

\begin{theorem}
\label{L21}
There exists a positive absolute constant $c$ and sequence of weights $w$ in $A_2^d$, such that $[w]_{A_2}^d\to \infty$ and at the same time

\textup(1\textup) $[w]_{A_2^d}\asymp \|M^d\|_{w^{-1}}^2$;

\textup(2\textup)  $\|S_{w}: L^{2} (w) \to L^2(w^{-1})\| \ge c \|M^d\|_{w^{-1}}\sqrt{\log \|M^d\|_{w^{-1}}}$;

\textup(3\textup)
$
\|S_{w}: L^{2,1} (w) \to L^2(w^{-1})\| \ge c  \|M^d\|_{w^{-1}}\sqrt{\log \|M^d\|_{w^{-1}}};
$

\textup(4\textup)
$\|S: L^{2} (w) \to L^2(w)\| \ge c \|M^d\|_{w}\sqrt{\log \|M^d\|_{w}}$.
\end{theorem}

\begin{proof}
The first  and the second claims were proved above, we just put $w=w_n$ and make parameter $p$ go to infinity. To see the third claim 
we just apply $S_{w}$, $w=w_n$, to one function $f=\1_{I_0}$ and use \eqref{Sbelow}. We have to observe also that for characteristic functions and any measure $\mu$, $L^2(\mu)$-norm and $L^{2,1}(\mu)$-norm coincide.

The fourth claim is just the second one applied to the inverse  $w=w_n^{-1}$. Of course this should be combined with the trivial remark that
$$
\|S_{w}: L^{2} (w) \to L^2(w^{-1})\| = \|S: L^{2} (w^{-1}) \to L^2(w^{-1})\|\,.
$$
\end{proof}

This theorem compares the rate norm of the weighted dyadic square function and weighted dyadic maximal function. We think that it gives the  new information on the relative strength of singularity of these operators.

 In terms of the comparison the  norm of the weighted dyadic square function with $A_2$ characteristic  $[w]_{A_2^{d}}$ the second claim does not give anything interesting because  by $(1)$
$$
 \|M^d\|_{w^{-1}}\sqrt{\log \|M^d\|_{w^{-1}}} \asymp  \sqrt{[w]_{A_2^{d}} \log [w]_{A_2^{d}}},
 $$
 and because $\|S_{w}: L^{2} (w) \to L^2(w^{-1})\|  \le [w]_{A_2^{d}}$, and this is known to be sharp.
 
 As we will see in Remark \ref{Nn} below $\|S_{w}: L^{2,1} (w) \to L^2(w^{-1})\| $ is typically  
 much smaller than $\|S_{w}: L^{2} (w) \to L^2(w^{-1})\| $. So in terms of the comparison the  norm 
 $\|S_{w}: L^{2,1} (w) \to L^2(w^{-1})\| $ with $A_2$ characteristic  $[w]_{A_2^{d}}$ the third 
 claim seems to be interesting, as it claims the existence of the sequence of weights with $[w]_{A_2^{d}}\to \infty$ and such that
 $$
\|S_{w}: L^{2,1} (w) \to L^2(w^{-1})\| \ge c  \sqrt{[w]_{A_2^{d}} \log [w]_{A_2^{d}}}\,.
$$
However, Remark \ref{Nn} below  also claims  that norms $\|S_{w}: L^{2,1} (w) \to L^2(w^{-1})\|$ and $\|S_{w}: L^{2} (w) \to L^2(w^{-1})\|$ are equivalent
as soon as the first one is $>>[w]_{A_2^d}^{1/2}$. This is exactly what happens in the case of weights $w=w_n$ built above and used in the proof of Theorem \ref{L21}.

\medskip

\begin{remark}
\label{Nn}
It seems reasonable to think that there exists an absolute constant $C$ such that
 $n\le N\le Cn$.
In fact, inequality $n\le N$ is obvious. On the other hand,
 \begin{align*}
& n\ge \sup_{I\in \cD} \|S_w \chi_I\|_{w^{-1}}/ \|\chi_I\|_{L^{2,1}(w)}=\\
& \sup_{I\in \cD} \|S_w \chi_I\|_{w^{-1}}/ \|\chi_I\|_{w}
\end{align*}
The latter quantity seems to be at least $\ge 
\frac1{C} \|S_w\|_{L^2(w)\to L^2(w^{-1})}=\frac1{C} N$ just by $T1$ theorem for dyadic square function operator.
In fact, such a $T1$ theorem for dyadic square function does not exist. 
 The  correct $T1$ theorem for dyadic square function \textup(see \cite{IVV1}\textup) looks as follows:
 \begin{equation}
 \label{T1SqF}
  \|S_w\|_{L^2(w)\to L^2(w^{-1})} \le C\max\big\{ \sup_{I\in \cD} \|S_w \chi_I\|_{w^{-1}}/ \|\chi_I\|_{w}, [w]_{A_2}^{1/2}\big\}\,.
  \end{equation}
  In our situation of $w=w_n$, where $w_n$ are the weights built above, the first term under the sign of {\it maximum} is dominating the second term. This in particular means that for $w_n$ built above the following holds:
  \begin{equation}
  \label{equiv21}
  \|S_{w_n}: L^{2} (w_n) \to L^2(w_n^{-1})\| \asymp \|S_{w_n}: L^{2,1} (w_n) \to L^2(w^{-1}_n)\| \asymp  \sqrt{[w_n]_{A_2^{d}} \log [w_n]_{A_2^{d}}}\,.
  \end{equation}
But we would like to repeat that the meaning of Theorem \ref{L21} is in comparison of $\|S\|_{w}$ with  $\|M\|_{w}$ and not with $[w]_{A_2^d}$.
  \end{remark}


\medskip

\subsection{A problem with duality} If the square function operator were a linear operator 
with symmetric or anti-symmetric kernel, then taking the adjoint, we would conclude from Theorem \ref{L21} that
\begin{equation}
\label{L2infty}
\|S_{w^{-1}_n}: L^{2} (w^{-1}_n) \to L^{2, \infty}(w_n)\| \ge c \sqrt{[w_n]_{A_2^{d}} \log [w_n]_{A_2^{d}}}\,.
\end{equation}
Now $f\to fw^{-1}$ is the isometry from $L^2(w^{-1})$ onto $L^2(w)$. Thus, \eqref{L2infty} would become
\begin{equation}
\label{L2}
\|S: L^{2} (w_n) \to L^{2,\infty}(w_n)\| \ge c \sqrt{[w_n]_{A_2^{d}} \log [w_n]_{A_2^{d}}}\,.
\end{equation}
This is precisely the missing estimate from below that would support the sharpness of the result of Domingo-Salazar--Lacey and Rey, \cite{DSaLaRey}.

Unfortunately, we do not know how to make this trick for square function operator. Henceforth, lines \eqref{L2infty}, \eqref{L2} are not proved.
They would complement very nicely the estimate of weak norm of the square function (from above) in Domingo-Salazar, Lacey, and Rey  \cite{DSaLaRey}.

\subsection{The estimate from above of $\|S_{w^{-1}_n}: L^{2} (w^{-1}_n) \to L^{2, \infty}(w_n)\| $ for weights built in the first sections}

Not only \eqref{L2infty} is not proved. It is, in fact, false for the weights $w_n$ built above, And, thus, an equivalent statement \eqref{L2}  is false too. We will show now that for $w=w_n$ built above one has an inequality quite opposite to the one we wish (that is \eqref{L2infty}). Namely,
\begin{equation}
\label{L2sqrt}
\|S_{w^{-1}_n}: L^{2} (w^{-1}_n) \to L^{2, \infty}(w_n)\| \le C\, \sqrt{[w_n]_{A_2^{d}}}\,.
\end{equation}
We will prove a more  general estimate now.

So below $w=w_n, p$ are those that has been constructed in the previous sections.
Consider operator   from $L^2(w^{-1})$
$$
\cS_{w^{-1}} f =\{(fw^{-1}, h_I)\frac{\chi_I(x)}{\sqrt{I}}\}_{I\in \cD}
$$
as an operator  from $L^2(w^{-1})$ to $L^{2, \infty}(\ell^2(\cD),w)$, which we want to estimate from below.

\bigskip

The pre-dual operator $S_w^*$ acts from $L^{2,1}(\ell^2(\cD), w)$ to $L^2(w^{-1})$. We would like to estimate its norm from below by
constructing a vector function 
$$
A(x):= \{a_I(x)\}_{I\in \cD}
$$
such that
\begin{equation}
\label{I00}
\sum_I a_I^2(x) = \chi_{I_0}\,,
\end{equation}
and
\begin{align}
\label{cSbelow}
\|\cS_{w}^* A\|^2_{w^{-1}} = \int_{I_0} \bigg|\sum_{I\in \cD} \sqrt{I}\La w\, a_I\Ra_I h_I(x)\bigg|^2 w^{-1}(x) dx \ge c \sqrt{[w]_{A_2^{d}} \log [w]_{A_2^{d}}}\, w (I_0)\,.
\end{align}

It is immediate that this inequality will bring \eqref{L2infty} (and equivalently \eqref{L2}).

\medskip

\subsection{However, \eqref{cSbelow} is impossible}
\label{imp}

\begin{remark}
The hope to prove \eqref{cSbelow} is futile.  In general, it seems like the approach used in \cite{LNO} and also used above to  estimate 
$$
\|\cS_w: L^{2,1}(w)\to L^{2}(w^{-1})\|
$$
from below  is not suitable for the estimate from below of $\|S_{w^{-1}}: L^2(w^{-1})\to L^{2, \infty}(w)\|=\|S: L^2(w) \to L^{2, \infty}(w)$.
\end{remark}

\medskip 

\begin{lemma}
\label{cA}
There exists a finite absolute constant $C$ such that for every vector function $A$ such that \eqref{I00} holds and for every weight $w$, such that
$$
[w]_{A_2^d} \asymp \|M^d\|_{w^{-1}}^2,
$$
we have
\begin{align}
\label{cSabove}
\|\cS_{w}^* A\|^2_{w^{-1}} = \int_{I_0} \bigg|\sum_{I\in \cD} \sqrt{I}\La w\, a_I\Ra_I h_I(x)\bigg|^2 w^{-1}(x) dx \le c \sqrt{[w]_{A_2^{d}}} \, w (I_0)\,.
\end{align}
\end{lemma}

\begin{proof}
Let us have the sequence of vector functions $A_n$ and weights $w_n\in A_2^d$ such that
\begin{align}
\label{cSbelow1}
\|\cS_{w}^* A_n\|^2_{w^{-1}} = \int_{I_0} \bigg|\sum_{I\in \cD} \sqrt{I}\La w\, a_I\Ra_I h_I(x)\bigg|^2 w^{-1}(x) / \left(\sqrt{[w_n]_{A_2^{d}}} \, w(I_0)\right)\to\infty\,.
\end{align}
If \eqref{cSbelow1} were true, then, by virtue  of $[w]_{A_2^d} \asymp \|M^d\|_{w^{-1}}^2$ we would have then that for the sequence of weights 
$$
\|S_{w^{-1}}: L^2(w^{-1})\to L^{2, \infty}(w) \| /\left( \|M^d\|_{w^{-1}} \right) \to \infty\,.
$$
But the last display formula  contradicts the following Theorem \ref{linThm}.

\end{proof}

\begin{theorem}
\label{linThm}
There exists an absolute constant $C$ such that or every weights $w$, the following holds
\begin{equation}
\label{lin}
\|S_{w^{-1}}: L^2(w^{-1})\to L^{2, \infty}(w)\|=\|S: L^2(w)\to L^{2, \infty}(w)\| \le C\|M^d\|_{w^{-1}}\,.
\end{equation}
\end{theorem}

\bigskip

\begin{remark}
Now it is clear why \eqref{cSbelow} cannot be true. The reason is the combination two facts proved for weights $w_n$ constructed above:
\begin{equation}
\label{Qp2eq}
1) \,\,[w_n]_{A_2^d} \asymp p^2\to \infty,\quad 2)\,\, \|M\|_{w_n^{-1}} \asymp p\,.
\end{equation}
\end{remark}



\medskip

The proof of Theorem \ref{linThm} is well-known. But we prove it here for the convenience of the reader.

\begin{proof}
Recall that for a function $g\in L^2(w^{-1})$
$$
\cR g :=\sum_{k=0}^\infty\frac{(M^d)^k g}{2\|(M^d)^k\|_{L^2(w^{-1})}}\,.
$$
From \eqref{cR} applied to $T=S$ we have for appropriate function $g$, $\|g\|_{w^{-1}} =1$, and any $F\in L^2(w^{-1})$
\begin{align*}
 \alpha(w\{S_{w^{-1}} F > \alpha\})^{1/2} \le 4 N(S, [\cR g]_{A_1^d}) \|F\|_{w^{-1}}\,,
 \end{align*}
where $N(S, [\cR g]_{A_1^d}) $ denotes the estimate from above of the weak norm $S: L^1(Rg)\to L^{1,\infty}(Rg)$. But it is well-known, see \cite{W}, page 39,  that
$$
N(S,[W]_{A_1^d})  \le C\, [W]_{A_1^d}\,,
$$
and from the definition of $\cR g$ it follows that
\begin{equation}
\label{Rg1}
[Rg]_{A_1^d} \le \|M^d : L^2(w^{-1})\to L^2(w^{-1})\|\,.
\end{equation}

Thus,
\begin{align}
\label{cRS}
 \alpha(w\{S_{w^{-1}} F > \alpha\})^{1/2} \le C \|M^d\|_{w^{-1}}\|F\|_{w^{-1}}\,,
 \end{align}
which is the claim of Theorem \ref{linThm}.

 \end{proof}
 
 \begin{remark}
 It is interesting to compare  two facts: one is that for all weights $w$ one has
 $$
 \|S_{w^{-1}}: L^2(w^{-1})\to L^{2, \infty}(w)\|=\|S: L^2(w)\to L^{2, \infty}(w)\| \le C\|M^d\|_{w^{-1}}\,,
 $$
 and another  fact is that there exists a sequence of weights with $ \|M^d\|_w$ tending to infinity such that
 $$
 \|S_w\|_{L^{2, 1}(w)\to L^2(w^{-1})}\ge  c \|M^d\|_{w^{-1}}\sqrt{\log \|M^d\|_{w^{-1}}}\
 $$
 The last line is just the paraphrase of Theorem  \ref{SM} if one takes Lemma \ref{Nn} into account.
 \end{remark}


\section{The proof of theorem \ref{Md}}
\label{MAX}

The proof  follows directly the steps of \cite{LNO}. It only is easier because the statement is dyadic.
We put initial numbers $\omega$, $\sigma$ to be
\begin{equation}
\label{initial}
\omega=1, \, \sigma=p\,.
\end{equation}

It is well known that for maximal function one has $T1$ theorem. Hence it is sufficient to check that with finite absolute constant $C$
\begin{equation}
\label{MT1}
\forall J \in \cD(I)\,\,\int_JM^d (w\chi_J)w^{-1} dx \le Cp^2 w(J)\,.
\end{equation}

Following \cite{LNO} define a function (of course we put $\omega=1$, but it is convenient to keep writing it):
\begin{equation}
\label{tilde}
\tilde w(x) :=\omega\sum_{\ell=1}^n 3^{\ell-1} \chi_{\supp w_{n-\ell+1}\setminus \supp w_{n-\ell}} + \omega 3^{n} \chi_{supp\, w_0}\,.
\end{equation}

\begin{lemma}
\label{maj}
With an absolute constant $C$, 
$M^d w\le C \tilde w$.
\end{lemma}

\begin{proof}
If $x\in \supp \, w_0$ we have (taking into account the normalization in \eqref{initial})
$$
w(x) =w_0(3^n\omega, \frac{\sigma}{3^n}; x) \le \frac{3^{n}}{\sqrt{p}} \left( (\sqrt{p} -\sqrt{p-1}) \chi_{I_-} +  (\sqrt{p} +\sqrt{p-1}) \chi_{I_+}\right)\le 2\cdot 3^n\,.
$$
On the complement of $\supp w_0$, $w\le \frac{3^n}{p}$.  So the claim of lemma is obvious for $x\in \supp w_0$.

If $x \in \supp w_{n-\ell+1}\setminus \supp w_{n-\ell}$, $\ell =1, 2, \dots, n$, then 
$$
w(x) \le \frac{3^{\ell-1}\om}{p},
$$
and outside of $\supp w_{n-\ell+1}$, $w(x)  \le \frac{3^{\ell-2}\om} {p}$. If we average $w$ 
over a dyadic interval  $J$ centered at $x \in \supp w_{n-\ell+1}\setminus \supp w_{n-\ell}$,  
and  if this $J$ intersects dyadic intervals forming $\supp w_{n-\ell}$, then each of this 
dyadic interval $I$ is inside $J$. Then, by Lemma \ref{av},
for each such $I$ we have $w(I) = 3^\ell |I|$. Combining all this together we get
$$
w(J) \le  2\cdot 3^\ell |J|\,,
$$
which proves the lemma.

\end{proof}

\begin{lemma}
\label{wnI0}
Let $I=[0,1]$ and $I_0= I, I_1= I^{++}, I_2= I_1^{++},\dots, I_{k-1}= I_{k-2}^{++}$ as in  Section \ref{spI}. Let $w=w_n$.
Then with a finite absolute constant $C$ we have
$$
\int_{I_0} \tilde w^2 w^{-1}\, dx \le Cp^2 w(I_0)\,.
$$
\end{lemma}
\begin{proof}
Let us consider first $m=0$.  Denote
$$
F_j= [0,1] \cap (\supp w_{n-j} \setminus \supp w_{n-j-1})\,.
$$
Let $A_j$ be the union of all special intervals  (see Section \ref{spI}) containing in $F_j$.

On the complement of $\supp w_{n-j-1}$, $w(x) \le \frac{3^j\om}{p}$. But in the points of this complement, which are not in $A_j$, one has  
$w(x) = \frac{3^j\om}{p}$. Hence, on such points $\tilde w(x) = p w(x)$. This is convenient, because together with $w(x)w^{-1}(x)=1$, this implies
$$
\int_{F_j\setminus A_j}  \tilde w^2 w^{-1} dx = p^2 \int_{F_j\setminus A_j}   w(x) dx = p^2 w(F_j\setminus A_j)\,.
$$

On $A_j$, $w^{-1}(x) = \frac{p}{3^j}\frac{1+5\eps}{9\eps}$. We saw in \eqref{A} that
$$
|A_j| =\left(\frac13(1-\frac{1}{4^{k-2}})\right)^j\frac2{4^k}\le \frac1{3^j} \frac2{4^k}\,.
$$
By definition of $\tilde w$ we conclude now
$$
\int_{\cup_{j=0}^{n-1} A_j} \tilde w^2 w^{-1} dx \le p\sum_{j=0}^{n-1} 3^{2j} \frac1{3^{2j}}\frac{2}{4^k}\frac{1+5\eps}{9\eps} \le \frac{n \,p}{4^k \eps}\,.
$$
As $n=4^k$ and $p\asymp \eps^{-1}$, we get by using Lemma \ref{av} in the last inequality
$$
\int_{\cup_{j=0}^{n-1} A_j} \tilde w^2 w^{-1} dx \le  \frac{n \,p}{4^k \eps}\le C p^2 w(I)\,.
$$
Therefore,
\begin{align*}
&\int_I \tilde w^2 w^{-1} dx \le \int_{\cup_{j=0}^{n-1} F_j\setminus A_j} \tilde w^2 w^{-1} dx + \int_{\cup_{j=0}^{n-1} A_j} \tilde w^2 w^{-1} dx +
\int_{\supp w_0} \tilde w^2 w^{-1} dx  \le \\
&  p^2 \sum_{j=0}^{n-1}w(F_j\setminus A_j) + Cp^2 w(I) +\int_{\supp w_0} \tilde w^2 w^{-1} dx  \le C p^2 w(I) + \int_{\supp w_0} \tilde w^2 w^{-1} dx 
\end{align*}
On $\supp w_0$, $\tilde w(x) = 3^n = w(x)$, hence the last integral is just at most $w(I)$. Finally we get ($I_0=I= [0,1]$)
\begin{equation}
\label{I0}
\int_{I_0} \tilde w ^2 w^{-1} dx \le Cp^2 w(I_0)=Cp^2 w(I)\,.
\end{equation}

\end{proof}

\medskip

\begin{lemma}
\label{wnIm}
Let $I=[0,1]$ and $I_0= I, I_1= I^{++}, I_2= I_1^{++},\dots, I_{k-2}= I_{k-3}^{++}$ as in  Section \ref{spI}. Let $w=w_n$.
Then with a finite absolute constant $C$ we have for $m=1, \dots, k-2$
$$
\int_{I_m} \tilde w^2 w^{-1}\, dx \le Cp^2 w(I_m)\,.
$$
\end{lemma}

\begin{proof}

Now we consider the case of $I_m$, $I=[0,1]$, $m=1,\dots, k-2$.
Notice that $I_m\setminus I_{m+1}$ consists of an interval $F:= I_{m}^{+-}$, which is one of the intervals forming $w_{n-1}$, and of 
interval $G$, such that $G$ belongs to $\supp w_n\setminus \supp w_{n-1}$. On such intervals, by the definition  \eqref{tilde} of $\tilde w$,  $\tilde w = \omega$, and $w=\frac{\omega}{p}$ (of course we can remember that $\omega$ is normalized in \eqref{initial}, but it does not matter in the calculations below). Thus on $G$, $\tilde w= p w$. 

Hence
$$
\int_{I_m\setminus I_{m+1}} \tilde w^2 w^{-1} dx \le p^2\int_G w^2w^{-1} dx + \int_{F} \tilde w^2 w^{-1} dx \le p^2 w(I) + \int_{F} \tilde w^2 w^{-1} dx\,.
$$

Notice that we can estimate the last integral by Lemma \ref{wnI0}. In fact, interval $F$ plays the role of $I$, and weight $w$ on $F$ (and so $\tilde w$) is constructed exactly as  $w=w_n$ on $I$, only it starts not with $\omega$ but with $3\omega$ and takes $n-1$ steps to be constructed. By scale invariance, we get by using Lemma \ref{wnI0}
$$
\int_{F} \tilde w^2 w^{-1} dx \le C p^2w(F) \le Cp^2 w(I)\,.
$$
Together two last display inequalities give
\begin{equation}
\label{Imm1}
\int_{I_m\setminus I_{m+1}} \tilde w^2 w^{-1} dx  \le (C+1)p^2 w(I_m\setminus I_{m+1})\,.
\end{equation}
We can write now ($m=1,\dots, k-2$) using \eqref{Imm1}:
$$
\int_{I_m}\tilde w^2 w^{-1} dx =\sum_{j=m}^{k-2} \int_{I_j\setminus I_{j+1}} \tilde w^2 w^{-1} dx + \int_{I_{k-1}} \tilde w^2 w^{-1} dx\le (C+1)p^2 w(I_m) +  \int_{I_{k-1}} \tilde w^2 w^{-1} dx\,.
$$

But on $I_{k-1}$ we have $\tilde w =\omega$  (see \eqref{tilde}) and so
\begin{equation}
\label{endp1}
\int_{I_{k-1}} \tilde w^2 w^{-1} dx \le C\omega^2 \frac{(1+5\eps)p}{9\eps\omega}|I_{k-1}|\le C\omega|I_{k-1}| p^2\,.
\end{equation}
On the other hand, as $m\le k-2$, we have by virtue of Lemma \ref{av}
\begin{equation}
\label{endp2}
w(I_m) \ge \omega |I_m| \ge \omega |I_{k-1}|\,.
\end{equation}

Combining \eqref{endp1} and \eqref{endp2} we get (for $m\le k-2$)
$$
\int_{I_{k-1}} \tilde w^2 w^{-1} dx \le C p^2 w(I_m)\,.
$$

We finally get that for $m=1,\dots, k-2$ the following holds
\begin{equation}
\label{Ik2}
\int_{I_m}\tilde w^2 w^{-1} dx \le Cp^2 w(I_m)\,.
\end{equation}

\end{proof}

In the next lemma we stop working with $\tilde w$.

\begin{lemma}
\label{wnIk1}
Let $I=[0,1]$ and $I_0= I, I_1= I_0^{++}, \dots, I_{k-1}= I_{k-2}^{++}$ as in  Section \ref{spI}. Let $w=w_n$.
Then with a finite absolute constant $C$ we have 
$$
\int_{I_{k-1}} (M^d (\chi_{I_{k-1}}w))^2 w^{-1}\, dx \le Cp^2 w(I_{k-1})\,.
$$
\end{lemma}

\begin{proof}
Clearly, by construction of $w$, we have
$$
M^d (\chi_{I_{k-1}}w) \le \frac{\omega}{p}\,.
$$
Hence,
$$
\int_{I_{k-1}} (M^d (\chi_{I_{k-1}}w))^2 w^{-1}\, dx \le \frac{\omega^2}{p^2}\frac{(1+5\eps)p}{9\eps\omega}|I_{k-1}| \le C\omega |I_{k-1}|\,.
$$
On the other hand,
$$
w(I_{k-1}) \ge \frac12\frac{\omega}{p} |I_{k-1}|\,.
$$
Therefore,
$$
\int_{I_{k-1}} (M^d (\chi_{I_{k-1}}w))^2 w^{-1}\, dx \le C p w(I_{k-1})\,.
$$

\end{proof}

\begin{lemma}
\label{wnIk1hat}
Let $I=[0,1]$ and $I_0= I, I_1= I_0^{++}, \dots, I_{k-1}= I_{k-2}^{++}$ as in  Section \ref{spI}. Let $J$ be the dyadic father of $I_{k-1}$, and $w=w_n$.
Then with a finite absolute constant $C$ we have 
$$
\int_{J} (M^d (\chi_{J}w))^2 w^{-1}\, dx \le Cp^2 w(J)\,.
$$
\end{lemma}

\begin{proof}
Interval $J$ consists of $J_-$ and $J_+ = I_{k-1}$. We have by Lemma \ref{av} and  by $w\le \frac{\omega}{p}$ on $I_{k-1}$ (see \eqref{wn}):
$$
x\in J_+=I_{k-1}\Rightarrow M^d (\chi_{J}w) \le \omega\,.
$$
And we also know by Lemma \ref{av} that $w(J) \ge w(J_-) =\frac12 \omega |J|$.
So
$$
\int_{J_+}(M^d (\chi_{J}w) )^2 w^{-1} dx \le \omega^2 \frac{(1+5\eps)p}{9\eps\omega} |J| \le C p^2 w(J)\,.
$$
Now notice that we can estimate the  integral  $\int_{J_-}(M^d (w\chi_{J}) )^2 w^{-1} dx$ by Lemma \ref{wnI0}. In fact, interval $J_-$ plays the role of $I$, and weight $w$ on $J_-$ (and so $\tilde w$) is constructed exactly as  $w=w_n$ on $I$, only it starts not with $\omega$ but with $3\omega$ and takes $n-1$ steps to be constructed. By scale invariance, we get by using Lemma \ref{wnI0}
$$
\int_{J_-} \tilde w^2 w^{-1} dx \le C p^2w(J_-) \le Cp^2 w(J)\,.
$$
Combining two last display inequalities we get the lemma's claim.

\end{proof}

Now we combine the lemmas of this Section with the fact that  for $x\in \supp w_{n-\ell+1}\setminus \supp w_{n-\ell}$, $w(x) \le \omega\frac{3^{\ell-1}}{p}$ to see that 
we obtained the following theorem.

\begin{theorem}
\label{sloi1}
Let $I=[0,1]$, $w=w_n$.  Let $J$ be any dyadic subinterval of $I$ such that it is not  contained in any dyadic interval
forming $\supp w_{n-1}$. There exists a finite absolute constant $C$ such that
\begin{equation}
\label{sloi1eq}
\int_J (M^d(\chi_J w))^2 w^{-1} dx \le C p^2 w(J)\,.
\end{equation}
\end{theorem}

Now we just notice that if interval $J$ is  inside or equal to one of the intervals forming $\supp w_{n-1}$, say $K$, then we can just notice that
$K$ plays the role of $I=[0,1]$ and $w_{n-1}$ is just the same type of weight as $w_n$, only starting with $3\omega$ instead of $\omega$, but the value of $\omega$ was immaterial in the above considerations.. Therefore, we can extend  the claim \eqref{sloi1eq} of Theorem \ref{sloi1} to dyadic intervals that are not contained in any dyadic interval
forming $\supp w_{n-2}$. The constant $C$ is exactly the same. We can continue to reason this way, and we obtain \eqref{MT1}.

\bigskip

\section{Direct square function operator on characteristic functions}
\label{chfun}

The reader can see that we always tried to approach the estimate of weighted square function operator norm via going 
to adjoint operator and estimating the adjoint operator on characteristic functions. A natural question arises why not to
try to prove the estimate \eqref{L2infty} by applying $S_{w^{-1}}$ to characteristic functions directly?

The answer is that to prove \eqref{L2infty} by this method would be impossible. This is because of

\begin{theorem}
\label{chfunThm}
\begin{equation}
\label{Lch}
\|S_{{w^{-1}}}\chi_I\|_{L^{2,\infty}(w)} \le C \sqrt{[w]_{A_2^d}}\, \|\chi_I\|_{w^{-1}},\,\, \forall I\in \cD\,.
\end{equation}
\end{theorem}

The proof takes the rest of this section. We introduce the following function of $3$ real variables
\begin{equation}
\label{supTest}
B_Q(u, v, \la):= \sup \frac{1}{|J|} w\left\{x\in J: \sum_{I\in D(J)} |\Delta_I w^{-1}|^2 \chi_I(x) >\lambda\right\} \,,
\end{equation}
where the supremum is taken over all $w\in A_2, [w]_{A_2}\le Q$, such that 
$$
\La w\Ra_J =u, \La w^{-1}\Ra_J = v\,.
$$
Notice that by scaling argument our function does not depend on $J$ but depends on $Q=[w]_{A_2}$. For brevity we can skip $Q$: $B:= B_Q$.

\begin{remark}
Ideally we want to find the formula for this function. Notice that this is similar to solving a problem 
of ``isoperimetric" type, where the solution of certain non-linear PDE is a common tool, see e. g. \cite{BarthH}, \cite{PIVO11}, \cite{INV}, \cite{IV}.
\end{remark}

\medskip

\subsubsection{Properties of $B$ and the main inequality}
\label{prop-weakT}
 Notice several properties of $B$:
\begin{itemize}
\item
$B$ is defined in  $\Omega:=\{(u,v, \la): 1\le uv \le Q, u>0, v>0, 0\le \la<\infty\}$.
\item If $P=(u,v, \la), P_+=(u_+,v_+, \la_+), P_-=(u_-,v_-,\la_-)$ belong to $\Omega$, and
$u=\frac12(u_+ +u_-)$, $v=\frac12(v_+ +v_-)$, $\la= \min(\la_+, \la_-)$,, then {\bf the main inequality} holds with constant $c=1$:
$$
B\left(u, v, \la+c (v_+ - v_-)^2\right)-\frac{B(P_+)+ B(P_-)}2 \ge 0.
$$
\item $B$ is decreasing in $\la$.
\item Homogeneity: $B(ut, v/t, \la/ t^2) = tB(u,v, \la), t>0$.
\item Obstacle condition: for all points $(u, v, \lambda)$ such that $10\le uv \le Q,\, \la\ge 0$, if $\la \le \delta \,v^2$ for a positive absolute constant $\delta$, one has $B(u, v, \la)= u$.
\item The boundary condition $B(u, v, \la) = 0$ if $uv=1$.
\end{itemize}

All these properties are very simple consequences of the definition of $B$. However, 
let us explain a bit the second and the fifth bullet. 
The second bullet is the consequence of the scale invariance of $B$. 
We consider data $P_+$ and find weight $w_+$ that almost supremizes $B(P_+)$.
By definition of $B$, we  have it on $J$. But by scale invariance we can think that $w_+$ lives on $J_+$. 
Then we consider data $P_-$ and find weight $w_-$ that almost supremizes $B(P_-)$.
Again we are supposed to have it on $J$. But by scale invariance we can think that $w_-$ lives on $J_-$. The next step is to consider the concatenation  of $w_+$ and $w_-$:
$$
w_c:=\begin{cases} w_+,\,\, \text{on} \,\, J_+\\
w_-,\,\, \text{on} \,\, J_-\,.\end{cases}
$$
Clearly this new weight is a competitor for giving the supremum for date $P$ on $J$. But it is only a competitor, the real supremum in \eqref{supTest} is bigger. This implies the second bullet above ({\it the main inequality}).

\bigskip

Now let us explain the fifth bullet above, we call it {\it the obstacle condition}. 
Let us consider a special weight $w_s$ in $J$: it is one constant on $J_-$ and just another constant on $J_+$. Moreover, we wish to have
$\La w_s^{-1}\Ra_{J_+} = 4 \La w_s^{-1}\Ra_{J_-}$. Notice that then $b\La w_s^{-1}\Ra_J \le |\Delta_J w_s^{-1}|$ 
with some positive absolute constant $b$.  

Now it is obvious that if $\lambda\le \delta^2\La w_s^{-1}\Ra_J^2$ 
then $\{x\in J: S^2_{w_s^{-1}} (\chi_J) \ge\lambda\} =J$ and so 
$\frac1{|J|} w_s\{x\in J: S^2_{w_s^{-1}} (\chi_J) \ge\lambda\} =\La w_s\Ra_J$.  
Notice now that $w_s$ is just one admissible weight, and that we have to 
take supremum over all such admissible weights. 

We get the fifth bullet above (=the obstacle condition):
$B(u, v, \la)=u$ for those points $(u, v, \la)$ in the domain of definition of $B$, 
where the corresponding $w_s$ with $\La w_s\Ra =u, \La w_s^{-1}\Ra_J= v$ exists. 
It is obvious that for all sufficiently large $Q$ and for any pair $(u,v)$ such that $10\le uv\le Q$ one can construct a just ``two-valued" $w_s$ as above with
$[w_s]_{A_2} \le Q$ (we recall that we deal only with dyadic $A_2$ weights).

Notice that the main inequality above transforms into a {\bf partial differential inequality} if considered infinitesimally (and if we tacitly assume that $B$ is smooth):
\begin{equation}
\label{infMI}
-\frac12 d^2_{u,v}B +c\frac{\partial B}{\partial \la} (dv)^2 \ge 0\,.
\end{equation}
We get it with $c=1$ for the function $B$ defined above  (if $B$ happens to be smooth).

We are not going to find $B$ defined in \eqref{supTest}, but instead we will construct smooth
$\cB$ that satisfies all the properties above (and of course \eqref{infMI}) except for the boundary condition 
(the last bullet above). It will satisfy even a slightly stronger properties, for example, the obstacle condition (the fifth bullet) will be satisfied with $1$ instead of $10$:
\begin{equation}
\label{obstTest}
\begin{aligned}
& \forall (u, v, \lambda)\,\text{ such that}\, 1\le uv \le Q,\, \la\ge 0, 
\\
& \text{if}\,\, \la \le \delta \,v^2\,\,\text{ for some}\, \delta>0, \,\text{then}\,\, B(u, v, \la)= u\,.
\end{aligned}
\end{equation}
Here $a$ will be some positive absolute constant (it will not depend on $Q$).

\bigskip

Using our usual telescopic sums consideration it will be very easy to prove the following

\begin{theorem}
\label{t-weakT}
Suppose we have a smooth function $\cB$ satisfying all the conditions above except the boundary condition, 
but satisfying the obstacle condition in the form \eqref{obstTest}. We also allow $c$ to be a small positive constant (say, $c=\frac18$). And suppose it also satisfies
\begin{equation}
\label{gr-weakT}
\cB(u, v, \la) \le A\,Q\frac{v}{\la}\,.
\end{equation}
Then \eqref{Lch} will be proved.
\end{theorem}

\begin{proof}
It is a stopping time reasoning. It is enough to think that $w$ is constant on some very small dyadic intervals and to prove the estimate on $C_{w,T}$ uniformly. Then we start with any such $w, [w]_{A_2}\le Q$, and we use the main inequality with $u=\La w\Ra_J, v=\La w^{-1}\Ra_J$, $u_\pm=\La w\Ra_{J_\pm}, v_\pm=\La w^{-1}\Ra_{J_\pm}$, 
$$ 
\la- c (\La w^{-1}\Ra_{J_+}-\La w^{-1}\Ra_{J_-})^2=: \la_{J_\pm}
$$
to obtain
\begin{equation}
\label{main-wT}
\begin{aligned}
 |J_+| \cB(\La w\Ra_{J_+}, \La w^{-1}\Ra_{J_+}, \la_{J_+})+ |J_-| \cB(\La w\Ra_{J_-}, \La w^{-1}\Ra_{J_-}, \la_{J_-}) \le  
 \\
 \cB(\La w\Ra_J, \La w^{-1}\Ra_J, \la)|J| \le \frac{A\,Q\La w^{-1}\Ra_J}{\la}|J|
\end{aligned}
\end{equation}
We continue to use the main inequality (because $J_\pm$ are not at all different from $J$) 
and finally after large but finite number of steps, on certain  collection $\cI$ of small intervals $I=J_{\pm\pm\cdots\pm}$ we come  to the situation that
\begin{equation}
\label{la-final}
\la_{J_{\pm\pm\cdots\pm}} <c (\La w^{-1}\Ra_{J_{\pm\pm\cdots+}}-\La w^{-1}\Ra_{J_{\pm\pm\cdots {-}}})^2\,.
\end{equation}
Collection $\cI$ may be empty of course, but we know that $I\in \cI$ if 
on $I$   the following holds for $x\in I$: 
$$
c\sum_{L\in D(J), I\subset L} |\Delta_I w^{-1}|^2 \chi_L(x)>\lambda\,.
$$
Let us combine \eqref{la-final} with an obvious inequality 
$$
c (\La w^{-1}\Ra_{J_{\pm\pm\cdots+}}-\La w^{-1}\Ra_{J_{\pm\pm\cdots {-}}})^2 \le \delta\La w^{-1}\Ra^2_{J_{\pm\pm\cdots}}\,.
$$
 At this moment we use the property 5 of $B$ called obstacle condition. 
 On intervals $I\in \cI$
the obstacle condition will provide us with $\cB(\La w\Ra_{I}, \La w^{-1}\Ra_{I}, \la_{I}) =\La w\Ra_I$.
So on a certain large  finite step $N$  we get from the iteration of \eqref{main-wT} $N$ times the following estimate
$$
\sum_{I\in D_N(J): I\in \cI} |I| \La w\Ra_I \le \frac{A\,Q\La w^{-1}\Ra_J}{\la}|J|\,.
$$
Therefore, we proved
$$
\frac1{|J|} w\{x\in J: \sum_{I\in D(J)} |\Delta_I w^{-1}|^2 \chi_I(x) >\lambda\} \le A\, Q \frac{\La w^{-1}\Ra_J}{\la}\,,
$$
which is \eqref{Lch}.

\end{proof}

\medskip

\subsubsection{Formula for the function $\cB$.\,  Monge--Amp\`ere equation with a drift.}
\label{formula-weakT}

Here is the formula for $\cB$ that satisfies all the properties in \ref{prop-weakT} (except for the last one, the boundary condition):
\begin{equation}
\label{form-weakT}
\begin{aligned}
& \cB(u, v,\la) =  \frac1{\sqrt\la} \Theta(u\sqrt\la, \frac{v}{\sqrt\la})\,,
\\
&\text{where}\,\,\, \Theta(\gamma, \tau):= \min \left(\gamma, Q e^{-\tau^2/2}\int_0^\tau e^{s^2/2} ds\right)\,.
\end{aligned}
\end{equation}

Notice that the fact that $\cB$ has the form $\cB(u, v, \la) = \frac1{\sqrt\la} \Theta(u\sqrt\la, \frac{v}{\sqrt\la})$ is trivial, this follows from property 4 called homogeneity.

Notice also that function $\Theta$ is  given in the domain enclosed by two hyperbolas
$$
H:=\{(\gamma, \tau)>0: 1\le \gamma \tau \le Q\}\,.
$$

All properties  listed at the beginning of Section \ref{prop-weakT} (except for the sixth bullet, which is boundary condition, but we do not use it anywhere) follow by direct computation.
In the next section we explain how to get this formula.

\subsubsection{Explanation of how to find such a function $\Theta$.}
\label{expl-weakT}

\bigskip

The main inequality (with $c=\frac18$) in terms of $\Theta$ becomes a ``drift concavity condition":
\begin{equation}
\label{Theta-MI}
\begin{aligned}
& \frac1{\sqrt{1+\frac{(\Delta\tau)^2}{8}}}\Theta \bigg( \sqrt{1+\frac{(\Delta\tau)^2}{8}} \frac{\gamma_-+\gamma_+}{2}, 
\frac1{\sqrt{1+\frac{(\Delta\tau)^2}{8}}} \frac{\tau_-+\tau_+}{2}\bigg)\ge 
\\
& \frac{\Theta(\gamma_-, \tau_-) + \Theta(\gamma_+, \tau_+)}{2}\,,
\end{aligned}
\end{equation}
where $(\gamma_-, \tau_-), (\gamma_+, \tau_+) \in H$, $0<\tau_-<\tau_+, \Delta \tau:=\tau_+-\tau_-$.

Assuming that $\Theta$ is smooth (we will find a smooth function), the infinitesimal version appears,  it is a sort of Monge--Amp\`ere relationship with a drift. Namely, the following matrix relationship must hold
\begin{equation}
\label{negative}
\begin{bmatrix}
\Theta_{\gamma\gamma}, & \Theta_{\gamma\tau}\\
\Theta_{\gamma\tau}, & \Theta_{\tau\tau} +\Theta + \tau\Theta_\tau -\gamma\Theta_\gamma
\end{bmatrix} \le 0\,.
\end{equation}
A direct calculation shows that this property is equivalent to the following one.
On any  curve  $\gamma =\phi(\tau)$ lying in the domain $H$ 
\begin{equation}
\label{change-var}
\begin{aligned}
&\text{and such that }\,\, \phi'' + \tau\phi' +\phi=0 
\\
& \text{we have}\, \, (\Theta(\phi(\tau), \tau))'' + \tau  (\Theta(\phi(\tau), \tau))' + \Theta(\phi(\tau), \tau)\le 0\,.
\end{aligned}
\end{equation}

This hints at a possibility to have  a change of variables $(\gamma, \tau)\to (\Gamma, T)$ such that
condition \eqref{negative} transforms to a simple concavity. To some extent this is what happens. Namely, notice the following simple
\begin{lemma}
\label{l-change-var}
Consider the following change of variable: $T=\int_0^\tau e^{s^2/2} ds$.  Then $\phi''(\tau)+\tau \phi'(\tau) +\phi(\tau)\le 0$ if and only if  $(e^{\tau^2/2} \phi(\tau))_{TT} \le 0$ and  $\phi''(\tau)+\tau \phi'(\tau) +\phi(\tau)= 0$ if and only if  $(e^{\tau^2/2} \phi(\tau))_{TT} = 0$.
\end{lemma}

\begin{proof}
The proof is a direct differentiation.
\end{proof}

This Lemma hints that the right change of variable should look like

\begin{equation}
\label{GaT}
\begin{cases} \Gamma:= \gamma e^{\tau^2/2},\\ T= \int_0^\tau e^{s^2/2} ds\,,\, (\gamma, \tau)\in \bR_+^1\times \bR_+^1\end{cases}
\end{equation}
then in the new coordinates the family of  curves $\gamma =\phi(\tau)$ such that $\phi'' + \tau\phi' +\phi=0 $ becomes a family of all  straight lines $\Gamma = CT +D$. (Notice that both families depend on two arbitrary constants.)

Denote
$$
O:= \{(\Gamma, T): (\gamma, \tau)\in G\}.
$$

The condition $ (\Theta(\phi(\tau), \tau))'' +\tau  (\Theta(\phi(\tau), \tau))' + \Theta(\phi(\tau), \tau)\le 0$ on any of these curves  becomes
\begin{equation}
\label{TT}
\left(e^{\tau^2/2}\Theta(\phi(\tau), \tau)\right)_{TT} \le 0\,, \,\, \Gamma=CT+D\,,\,\, (\gamma, \tau)\in G\,.
\end{equation}
which is the concavity of $e^{\tau^2/2}\Theta (\gamma, \tau)$  in a new coordinate $T$ along the line $\Gamma = CT +D$.
Let us rewrite two functions in the new coordinates:
$$
\Phi(\Gamma, T):= \Theta(\gamma, \tau)\,,\,\, U(T) := e^{\tau^2/2}\,.
$$
Then \eqref{TT} transforms into
\begin{equation}
\label{concUPhi}
\forall C, D \in \bR\,,\,\,\left(U(T)\Phi(CT+D, T))\right)_{TT} \le 0\,, \,\, (\Gamma, T)\in O\,.
\end{equation}
This is just a concavity of $U(T) \Phi(\Gamma, T)$ on $O$ of course. Notice that neither $H$ nor $O$ are convex, so we  should understand \eqref{concUPhi} as a local concavity in $O$: just the negativity of its second differential form 
$$
d^2_{\Gamma, T} (U(T)\Phi(\Gamma, T)\le 0\,,\,\, (\Gamma, T)\in O\,.
$$

\medskip

So we reduce the question to finding a concave function in new coordinates. Now we choose a simplest possible concave function:
$$
U(T) \Phi(\Gamma, T) := \min (\Gamma, KT)\,,
$$
where the constant $K=K(Q)$ will be chosen momentarily.

\medskip

If we write down now $\Theta(\gamma, \tau)= \Phi(\Gamma, T)$ in the old coordinates, we get exactly function $\Theta$ from \eqref{form-weakT} (we need to define constant $K$ yet), namely,
\begin{equation}
\label{Theta5}
\Theta(\gamma, \tau):=\min \left(\gamma, K e^{-\tau^2/2}\int_0^\tau e^{s^2/2} ds\right)\,.
\end{equation}

Recall that now we can consider
\begin{equation}
\label{Btest}
B(u, v, \la)=\frac1{\sqrt{\la}} \Theta(u\sqrt{\la}, \frac{v}{\sqrt{\la}})
\end{equation}
and we are going to apply Theorem \ref{t-weakT} to it. But we need to choose $K$ to satisfy all the conditions (except the last one) at the beginning of Section \ref{prop-weakT}.

First of all it is now very easy to understand why the form of the domain $H=\{1\le \gamma \tau\le Q\}$ plays the role. In fact, by choosing
$$
K=AQ
$$
with some absolute constant $A$, we guarantee that in this domain our function $\Theta$ satisfies the obstacle condition
\begin{equation}
\label{obstacle-weakT}
\Theta(\gamma, \tau)  = \gamma\,\, \text{as soon as}\,\, \tau\ge a_0>0\,,
\end{equation}
where $a_0$ is an absolute positive constant.
In fact, for all sufficiently small $\tau$, $e^{-\tau^2/2}\int_0^\tau e^{s^2/2} ds \asymp \tau$, and therefore, 
for all sufficiently small $\tau$ (smaller than a certain absolute constant)
$$
\Theta(\gamma, \tau):=\min \left(\gamma, K \tau\right)\,.
$$
The fifth condition at the beginning of Section \ref{prop-weakT} (the obstacle condition) requires then that $\min \left(\gamma, K \tau\right)=\gamma$ if $\tau\ge a_0>0$. But on the upper hyperbola then $\gamma=Q/a_0$ for $\tau=a_0$. We  see that the smallest possible $K$ we can choose to satisfy the obstacle condition is $K\asymp Q$.

\bigskip

Secondly, function $\Theta$ satisfies the infinitesimal condition \eqref{negative} by construction. But we need to check that the main inequality \eqref{Theta-MI} is satisfied as well.

This can be done by the following lemma.

\begin{lemma}
\label{redu-weakT}
Inequality \eqref{Theta-MI} for function $\Theta$ built above holds if and only if the following inequality is satisfied for $\phi(\tau):=  e^{-\tau^2/2} \int_0^\tau e^{s^2/2} ds$:
\begin{equation}
\label{phi}
\begin{aligned}
& \frac1{\sqrt{1+\frac{(\Delta\tau)^2}{100}}}\phi \bigg(\frac{\tau_1+\tau_2}{2\sqrt{1+\frac{(\Delta\tau)^2}{100}}}\bigg) \ge 
\\
& \frac{\phi(\tau_1) +\phi(\tau_2)}{2}\,, \, \forall \,0<\tau_1\le\tau_2 \le\tau_0\,,
\end{aligned}
\end{equation} 
with some absolute positive small constant  $\tau_0$.
\end{lemma}

\begin{proof}
Lemma is easy, because we can immediately see that the main inequality \eqref{Theta-MI} commutes with the operation of minimum. 
\end{proof}

 As soon as   \eqref{phi} is checked, inequality \eqref{Lch} and Theorem \ref{chfunThm} are proved. There are many ways  to prove  \eqref{phi}, we choose the proof that imitates (with some changes) the proof of Barthe and Maurey of the similar statement, see \cite{BM}.

\begin{proof}

Consider the
new function
$$
U(p, q):= \frac1q\phi(\frac{p}q)\,, 
$$
given in the domain $\{(p, q): p\ge 0, 0\le \frac{p}q \le \tau_0\}$. Here $\tau_0$ is a small positive number, say $\tau_0=0.001$.
Then \eqref{phi} follows from
\begin{equation}
\label{U}
U(p, q) \ge \frac12 U(p+a, \sqrt{q^2-\frac{a^2}{100}}) + \frac12 U(p-a, \sqrt{q^2-\frac{a^2}{100}})\,.
\end{equation}
Notice that 
\begin{equation}
\label{apq}
a\le p \le \frac1{1000} q,
\end{equation}
and so
\begin{equation}
\label{09}
\sqrt{q^2-\frac{a^2}{50}}/\sqrt{q^2-\frac{a^2}{100}} \le 0.9\,.
\end{equation}
Notice also that infinitesimally \eqref{phi} and \eqref{U} are satisfied, this is very easy to see because $\phi$ is strictly concave on small interval $[0, \tau_0]$.

Without loss of generality assume $a\geq 0$. Consider the process 
\begin{align*}
X_{t} = U(p+B_{t}, \sqrt{q^{2}- \frac{t}{50}}),  \quad 0\le  t\leq q^2/4.
\end{align*}

Here $B_{t}$ is the standard Brownian motion starting at zero. 
The infinitesimal version of \eqref{U} shows that
$$
\frac12 U_{pp} -\frac1{100} \frac{U_q}{q} \le0\,.
$$
It follows from Ito's formula  together with the last observation that $X_{t}$ is a supermartingale. 
Let  $\tau$ be the stopping time
\begin{align*}
\tau = \frac{q}{2}\wedge\inf\{ t \geq 0 : B_{t} \notin (-a,a)\}.
\end{align*} 

It follows from the fact that $X_t$ is a supermartingale that  
\begin{align*}
&U(p,q) = X_{0} \geq \mathbb{E} X_{\tau}=\mathbb{E} U(p+B_{\tau}, \sqrt{q^{2}+\tau/50}) = \\
&P(B_{\tau}=-a) \mathbb{E}(U(p-a, \sqrt{q^{2}-\tau/50}) | B_{\tau}=-a)+
P(B_{\tau}=a) \mathbb{E}(U(p+a, \sqrt{q^{2}+\tau/50}) | B_{\tau}=a)+\\
&P(|B_{\tau}|<a, \tau=q) \mathbb{E}(U(p+a, \sqrt{q^{2}+\tau/50}) | B_{\tau}=a)\,.
\end{align*}

Notice that the last probability is very small. In fact,
$$
\frac14q^2P(|B_{\tau}|\le a, \tau\ge \frac{q}{2}) \le {\mathbb E} \tau \le \mathbb{E} |B_\tau|^2 \le a^2\le \frac1{1000} q^2\,.
$$
Also clearly $P(B_{\tau}=-a) = P(B_{\tau}=-a) $, and by the last observation this probabilities are at least $\frac1{2.2}$. So we obtained
\begin{align*}
&U(p,q) \ge \\
&\frac{1}{2.2}\left(\mathbb{E}(U(p-a, \sqrt{q^{2}-\tau/50}) | B_{\tau}=-a)+
\mathbb{E}(U(p+a, \sqrt{q^{2}-\tau/50}) | B_{\tau}=a) \right) \geq \\
&\frac{1}{2.2}\left(U\left( p-a, \sqrt{q^{2}- \mathbb{E}(\tau/50|B_{\tau}=-a) }\right)+U\left( p+a, \sqrt{q^{2}- \mathbb{E}(\tau/50|B_{\tau}=a) }\right) \right)=\\
&\frac{1}{2.2}\left( U\left(p-a, \sqrt{q^{2}-a^{2}/50}\right) +  U\left(p+a, \sqrt{q^{2}-a^{2}/50}\right) \right)\,.
\end{align*}

Notice that  we have used   $\mathbb{E}(\tau | B_{\tau}=a) =\mathbb{E}(\tau | B_{\tau}=-a)=a^{2}$, and the fact that the map $t \mapsto U(p, \sqrt{t})$, $t\in [q^2/4, q^2]$,  is convex together with Jensen's inequality.   The convexity follows from the fact that
$$
t\to \frac1{\sqrt{t}} \phi(\frac{p}{\sqrt{t}})
$$ 
is convex if $t \in [q^2/4, q^2]$. This is easy to check by direct calculation, putting $x=\frac{p}{\sqrt{t}}$ we get
$$
\bigg[\frac1{\sqrt{t}} \phi(\frac{p}{\sqrt{t}})\bigg]''_{tt} = \frac34\frac1{t^{5/2}}[(1-2x^2-x^2(1-x^2))\phi(x) + 2x -x^3] >0,
$$
if $x=\frac{p}{\sqrt{t}}\le 2\frac{p}q \le 2\tau_0$ is sufficiently small.

Now we use \eqref{09} to conclude that
\begin{align*}
&U(p,q) \ge  \frac{1}{2.2}\left( U\left(p-a, \sqrt{q^{2}-a^{2}/50}\right) +  U\left(p+a, \sqrt{q^{2}-a^{2}/50}\right) \right) \ge \\
&1.1 \frac{1}{2.2}\left( U\left(p-a, \sqrt{q^{2}-a^{2}/100}\right) +  U\left(p+a, \sqrt{q^{2}-a^{2}/100}\right) \right)\,,
\end{align*}
where we obtained the last inequality by denoting $x_2= \frac{p\pm a}{\sqrt{q^2-\frac{a^2}{50}}}$, $x_1= \frac{p\pm a}{\sqrt{q^2-\frac{a^2}{100}}}$, and noticing that
$$
\phi(x_2)/\phi(x_1) \ge 1\ge 0.99 \ge 1.1 \times 0.9 \ge 1.1 x_1/x_2 =  1.1 \sqrt{q^2-\frac{a^2}{50}}/ \sqrt{q^2-\frac{a^2}{100}}\,.
$$
Here we use \eqref{apq} again.

\end{proof}


\begin{thebibliography}{999}

\bibitem{BarthH} {\sc F. Barthe, N.~Huet}, {\em On Gaussian Brunn--Minkowskii inequalities.} Stud. Math. {\bf 191}, (2009), pp. 283--304.

\bibitem{BM} {\sc F.~Barthe, B.~Maurey}, {\em Some remarks on isoperimetry of Gaussian type}, Annales de l'Institut Henri Poincare (B) Probability and Statistics, Vol. 36, Iss. 4, pp. 419--434.

\bibitem{Buck}{\sc St. Buckley}, 
{\em Summation condition on weights}, Mich. Math. J., {\bf 40} (1993), 153--170.

\bibitem{Buck1}{\sc St. Buckley}, {\em Estimates for operator norms on weighted spaces and reverse Jensen inequalities}.
Trans. Amer. Math. Soc. 340 (1993), no. 1, 253--272.

\bibitem{DSaLaRey}. C. Domingo-Salazar; M. Lacey; G. Rey,  \emph{Borderline weak-type estimates for singular integrals and square functions}. Bull. Lond. Math. Soc. 48 (2016), no. 1, 63--73. 

\bibitem{IV} P.~Ivanisvili, A.~Volberg, \emph{Isoperimetric functional inequalities via the maximum principle: the exterior differential systems approach}, arXiv: 1511.06895.

\bibitem{PIVO11} P.~Ivanisvili, A.~Volberg,  \emph{Improving Beckner's bound via Hermite functions},  Anal. PDE 10 (2017), no. 4, 929--942.

\bibitem{IVV1} P. Ivanisvili, F. Nazarov, A. Volberg,  \emph{Strong weighted and restricted weak weighted  estimates of the square function revisited} Preprint, 2018, pp. 1--18.

\bibitem{IVV2} P. Ivanisvili,  A. Volberg,  \emph{
Martingale transform and Square function: some end-point weak weighted estimates}, arXiv:1711.10578 .


\bibitem{INV} P. Ivanisvili, F. Nazarov, A. Volberg, \emph{Hamming cube and martingales}, C. R. Acad. Sci. Paris, Ser. I 355 (2017) 1072--1076.


\bibitem{INV2} P. Ivanisvili, F. Nazarov, A. Volberg, \emph{Square function and the Hamming: duality}, arXiv:1706.01930, pp.1--13.






 \bibitem{LNO} {\sc. A. Lerner, F. Nazarov, S. Ombrosi}, 
{\em On the sharp upper bound related to the weak Muckenhoupt-Wheeden conjecture}, arXiv:1710.07700, pp. 1--17.

\bibitem{LOP1}
A.K. Lerner, S. Ombrosi and C. P\'erez,
{\it Sharp $A_1$ bounds for Calder\'on-Zygmund operators and the relationship with a problem of Muckenhoupt and Wheeden}, Int. Math. Res. Not. IMRN  2008,  no. 6, Art. ID rnm161, 11.

\bibitem{LOP2}
A.K. Lerner, S. Ombrosi and C. P\'erez, {\it $A_1$ bounds for Calder\'on-Zygmund operators related to a problem of Muckenhoupt and Wheeden}, Math. Res. Lett. {\bf 16} (2009),  no. 1, 149--156.

\bibitem{M}
K. Moen, {\it Sharp one-weight and two-weight bounds for maximal
operators}, Studia Math. {\bf 194} (2009), no. 2, 163--180.

\bibitem{NRVV1}
F. Nazarov, A. Reznikov, V. Vasyunin and A. Volberg,
{\it A Bellman function counterexample to the $A_1$ conjecture: the blow-up of the weak norm estimates of weighted singular operators},
preprint (2015). Available at https://arxiv.org/abs/1506.04710.

\bibitem{NRVV2}
F. Nazarov, A. Reznikov, V. Vasyunin and A. Volberg,
{\it On weak weighted estimates of martingale transform}, preprint (2016). Available at https://arxiv.org/abs/1612.03958.

\bibitem{NTV}
F. Nazarov, S. Treil and A. Volberg, {\it The Bellman functions and two-weight inequalities for Haar multipliers}, J. Amer. Math. Soc. {\bf 12} (1999), no. 4, 909--928.


\bibitem{OsW} {\sc A. Osekowski}, {\em Weighted square function inequalities}, Preprint,  pp. 1--15.



\bibitem{R}
M.C. Reguera, {\it On Muckenhoupt-Wheeden conjecture}, Adv. Math., {\bf 227} (2011), no. 4, 1436--1450.

\bibitem{RT}
M.C. Reguera and C. Thiele, {\it The Hilbert transform does not map $L^1(Mw)$ to $L^{1,\infty}(w)$}, Math. Res. Lett. {\bf 19} (2012),  no. 1, 1--7.


\bibitem{W} M. Wilson, \emph{Littlewood--Paley Theory and Exponential Integrability}, Springer, 2006, pp. 1--196.




\end{thebibliography}
\end{document}